\documentclass[10 pt]{article}
 \pdfoutput=1
\usepackage{setspace}

\usepackage[T1]{fontenc}
\usepackage[dvips]{graphicx}
\usepackage{epsfig}
\usepackage{amsmath,amssymb,amscd,amsthm}
\usepackage{subfig}
\usepackage[latin1]{inputenc}
\usepackage{latexsym}
\usepackage{graphics}
\usepackage{colortbl}
\usepackage{color}
\usepackage{ae}
\usepackage{enumitem}
\usepackage{mathabx}
\usepackage[rightcaption]{sidecap}
\graphicspath{{Fig/}}
\def\H{{\mathbb H}}

\usepackage{textcomp}
\begin{document}

\setcounter{secnumdepth}{3}
\setcounter{tocdepth}{2}

\newtheorem{definition}{Definition}[section]
\newtheorem{lemma}[definition]{Lemma}
\newtheorem{sublemma}[definition]{Sublemma}
\newtheorem{corollary}[definition]{Corollary}
\newtheorem{proposition}[definition]{Proposition}
\newtheorem{theorem}[definition]{Theorem}

\newtheorem{remark}[definition]{Remark}
\newtheorem{example}[definition]{Example}

\newcommand{\mf}{\mathfrak}
\newcommand{\mb}{\mathbb}
\newcommand{\ol}{\overline}
\newcommand{\la}{\langle}
\newcommand{\ra}{\rangle}
\newcommand{\hess}{\mathrm{Hess}}

\newtheorem{thmprime}{Theorem}
\renewcommand{\thethmprime}{1.\arabic{thmprime}\textquotesingle}

\renewcommand{\P}{\mathcal{P}}


\newtheorem{Alphatheorem}{Theorem}
\renewcommand{\theAlphatheorem}{\Alph{Alphatheorem}}

\newtheorem{Alphatheoremprime}{Theorem}
\renewcommand{\theAlphatheoremprime}{\Alph{Alphatheoremprime}\textquotesingle}

\theoremstyle{definition}

\theoremstyle{plain}
\newtheorem*{flp}{Flip algorithm}
\newtheorem{clm}{Claim}

\newcommand{\EM}{\ensuremath}
\newcommand{\norm}[1]{\EM{\left\| #1 \right\|}}

\newcommand{\modul}[1]{\left| #1\right|}

\def\co{\colon\thinspace}
\def\I{{\mathcal I}}
\def\N{{\mathbb N}}
\def\R{{\mathbb R}}
\def\Z{{\mathbb Z}}
\def\Sph{{\mathbb S}}
\def\Tor{{\mathbb T}}
\def\Disk{{\mathbb D}}
\def\Fl{{\mathbb F}}

\def\H{{\mathbb H}}
\def\RP{{\mathbb R}{\mathrm{P}}}
\def\dS{{\mathrm d}{\mathbb{S}}}

\def\phi{\varphi}
\def\epsilon{\varepsilon}
\def\V{{\mathcal V}}
\def\E{{\mathbb E}}
\def\F{{\mathcal F}}
\def\C{{\mathcal C}}
\def\K{{\mathcal K}}

\renewcommand{\span}{\operatorname{span}}
\newcommand{\ang}{\operatorname{ang}}
\newcommand{\coarea}{\operatorname{coarea}}
\newcommand{\cone}{\operatorname{cone}}
\newcommand{\pr}{\operatorname{pr}}
\newcommand{\vol}{\operatorname{vol}}
\newcommand{\covol}{\operatorname{covol}}
\newcommand{\st}{\operatorname{st}}
\newcommand{\ost}{\operatorname{st}^\circ}
\newcommand{\inn}{\operatorname{int}}
\renewcommand{\d}{\operatorname{d}}

\def\dist{\mathrm{dist\,}}
\def\diam{\mathrm{diam\,}}

\def\M{{\mathcal M}}
\def\Mt{{\mathcal M}_{\mathrm{tr}}}
\def\T{{\mathcal T}}
\def\Vt{V_{\mathrm{tr}}}

\begin{footnotesize}
\title{Polygons of the Lorentzian plane and spherical simplexes}
\author{Fran\c{c}ois Fillastre\\
University of Cergy-Pontoise\\ UMR CNRS 8088\\Departement of Mathematics\\
F-95000 Cergy-Pontoise\\ FRANCE \\
francois.fillastre@u-cergy.fr}

\date{(v3) \today}
\end{footnotesize}
\maketitle

\section{Introduction}

It is a common occurence that sets of geometric objects themselves
carry some kind of geometric structure. A classical example for this is
the set of all conformal structures on a given compact surface. Riemann
discovered that this set, the ``space'' of conformal structures,
 can be described by a finite number of parameters
called {\it moduli}. The corresponding {\it parameter} or {\it moduli space} turned
out to be a very interesting geometric object in itself
whose study is the subject of Teichm\"uller theory.

On a more basic level, one can consider spaces consisting of objects
of elementary geometry like (shapes of) polyhedra in Euclidean space.
Thurston \cite{T} found that in this case, the corresponding moduli space
carries the structure of a complex hyperbolic manifold,  and he established
a link with sets of triangulations of the 2-sphere.

Bavard and Ghys \cite{BG92} considered sets of polygons in the Euclidean
plane. Fix a compact convex polygon $P$ with $n\geq3$ edges and let
${\cal P}(P)$ be the space of convex polygons with $n$ edges parallel to
those of $P$. The elements of ${\cal P}(P)$ are then determined by the distances of the lines
containing the edges from the origin, which gives $n$ parameters. Following [Thu98],
Bavard and Ghys proved that on the space of parameters, the area
of the polygons in ${\cal P}(P)$ is a quadratic form, and they computed its signature.
The kernel of the corresponding bilinear form has dimension 2
(due to the fact that area is invariant under translations), and there is only one
positive direction. Hence, up to the kernel, one gets a Lorentzian signature.
As a consequence, the set of elements of ${\cal P}(P)$ with area equal to one,
considered up to translations, can be identified with a subset of the
hyperbolic space $\mathbb{H}^{n-3}$. This subset turns out to be
a hyperbolic convex polyhedron of a special kind: it is a simplex with the property
that each hyperplane containing a facet meets orthogonally all but two hyperplanes containing
the other facets. Such simplices are called {\it hyperbolic orthoschemes}.
The dihedral angles of the orthoscheme can be computed from the angles of
$P$, and \cite{BG92} contains a list of convex polygons $P$ such that the orthoscheme
obtained from $P$ is of Coxeter type, i.e. has acute angles of the form $\pi/k$,
$k\in\mathbb{N}$. This list  was previously known \cite{ImHof2,ImHof1}, but it appeared it
was incomplete \cite{F}.

In this paper we consider a class of non-compact plane polygons whose moduli
space is a {\it spherical} orthoscheme. These polygons, the $t$-convex polygons
introduced in Section~\ref{sec:tcvxe}, are best described not in terms of the Euclidean geometry
on $\mathbb{R}^2$, but as subsets of the Lorentz plane. Instead of the area we will
consider a suitably defined coarea that turns out to be a positive definite
quadratic form on the parameter space, an $n$-dimensional vector space.
Restricting to coarea one we obtain a subset of the unit sphere in that
parameter space, and this subset is shown to be a spherical orthoscheme.
Moreover, any spherical orthoschem can be obtained in this way.

It is amusing that in \cite{BG92} Euclidean polygons led to Lorentz metrics
and hyperbolic orthoschemes, while in the present paper Lorentzian
polygons give rise to Euclidean metrics and spherical orthoschemes.
The author does not know if there is a way to obtain {\it Euclidean}
orthoschemes from spaces of plane convex polygons.

\section{Background on the Lorentz plane}

Recall that the {\it Lorentz plane} is $\mathbb{R}^2$ equipped with the
{\it Lorentz inner product}, that is the bilinear form $\langle \binom{x_1}{x_2} ,\binom{y_1}{y_2} \rangle_1=x_1y_1-x_2y_2. $
A non-zero vector $v$ can be \emph{space-like}
($\langle v,v\rangle_1>0$), \emph{time-like} ($\langle v,v\rangle_1<0$) or \emph{light-like}
($\langle v,v\rangle_1=0$).
The set of time-like vectors 
has two connected components, and we denote the upper one, the set of 
\emph{future} time-like vectors, by
$$\F:=\{x\in\R^2\vert \langle x,x\rangle_1<0, x_2>0 \}. $$
The set of unit future time-like vectors is
$$\H:=\{x\in \R^2| \langle x,x\rangle_1=-1, x_2>0 \},$$
which will be the analog of the circle in the Euclidean plane, see Figure~\ref{fig:futcon}.
In higher dimension, the generalization of $\H$ together with its induced metric is a model of the hyperbolic space, 
in the same way that
the unit sphere for the Euclidean metric with its induced metric is a model of the round sphere.
 In particular, if the angle between two unit vectors in the Euclidean plane is seen as
the distance between the two corresponding points on the circle, 
 the  \emph{(Lorentzian) angle} between two future time-like vectors $x$ and $y$ is the unique $\varphi> 0$ such that
\begin{equation}\label{eq:angle}
 \cosh \varphi=-\frac{ \langle x,y \rangle_1}{\sqrt{\langle x,x\rangle_1\langle y,y\rangle_1}}
\end{equation}
(see \cite[(3.1.7)]{Rat06} for the existence of $\varphi$).
The angle $\varphi$ is the distance on $\H$ (for the induced metric) between $x/\sqrt{-\langle x,x\rangle_1}$ 
and $y/\sqrt{-\langle y,y\rangle_1}$.

\begin{figure}[h]
 \centering
 \includegraphics[scale=0.3]{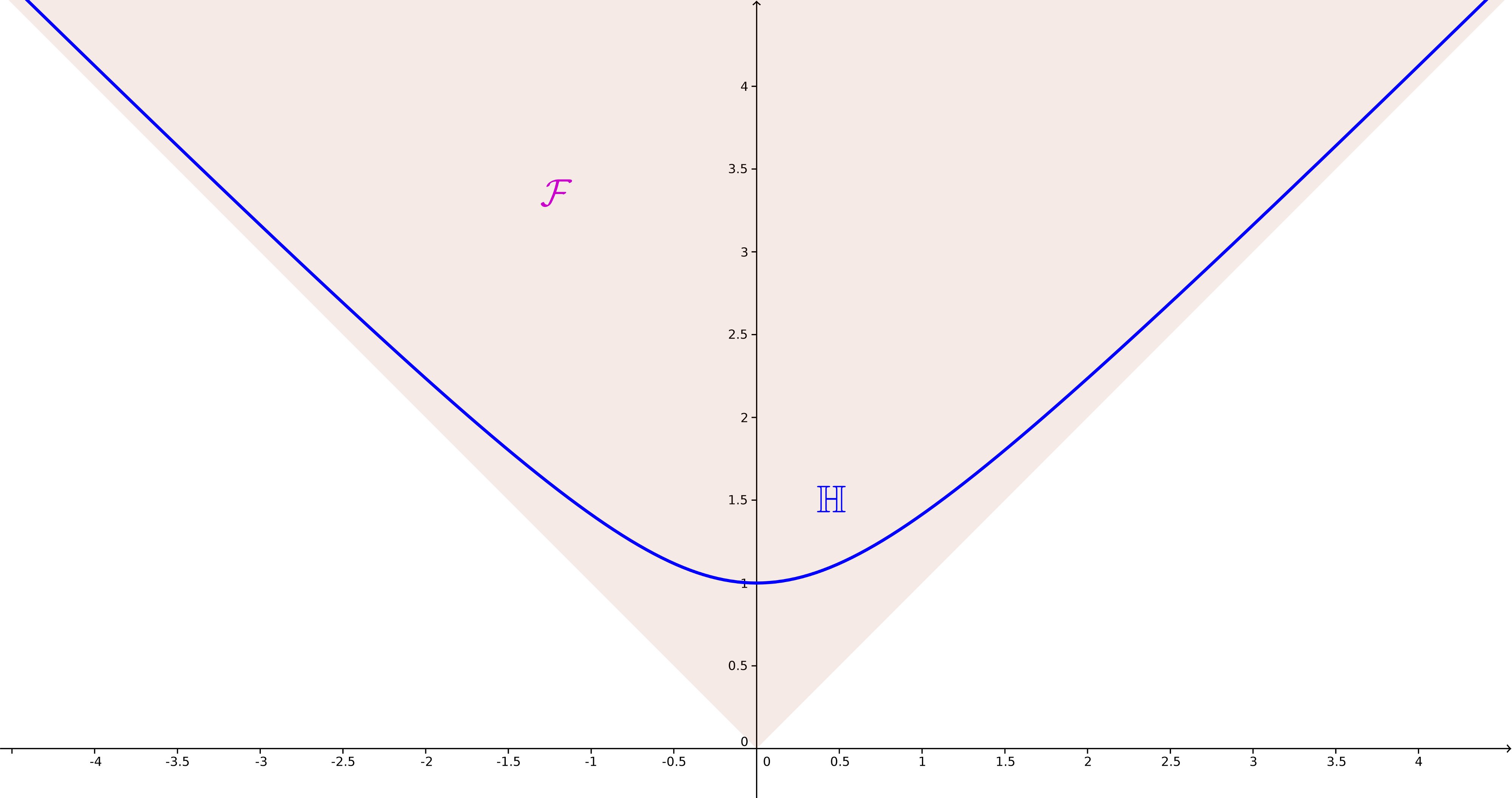}  
 \caption{The cone $\F$ of future time-like vectors and 
 the curve $\H$ of unit future time-like vectors.\label{fig:futcon}
 } 
   \end{figure}

$\F$  and $\H$ are globally invariant under the action of the linear isometries of the Lorentzian plane, called 
\emph{hyperbolic translations}:
\begin{equation}\label{eq:hyp trans}H_{t}:= \left(
    \begin{array}{cc}
       \cosh t  & \sinh t \\
   \sinh t & \cosh t
    \end{array}
\right), t\in\R.\end{equation}
In all the paper we fix a positive $t$. 
We denote by $<H_t>$ the free group
spanned by $H_t$.


\section{$t$-convex polygons}\label{sec:tcvxe}

Let $a\in\F$. We will denote by
\begin{equation*}\label{eq def affine}
a^{\bot}:=\{x\in\R^{2}| \langle x,a\rangle_1=\langle a,a\rangle_1 \}
\end{equation*}
the line
that passes through $a$ and is parallel to the $1$-dimensional subspace
orthogonal to $a$ under $\langle \cdot,\cdot\rangle_1$.

\begin{definition}\label{def:gen}
Let $(\eta_1,\ldots,\eta_n)$, $n\geq 1$, be pairwise distinct unit future time-like vectors in the Lorentzian plane
(i.e.~$\eta_i\in\H$),
 and let $h_1,\ldots,h_n$ be positive numbers. 
 A \emph{$t$-convex  polygon} $P$ is  the intersection of the half-planes
bounded by the lines
$$(H_t^k (h_i\eta_i))^{\bot}, \forall k\in\mathbb{Z}, \forall i=1,\ldots,n.$$
The half-planes are chosen such that the vectors $\eta_i$ are inward pointing.
The positive numbers $h_i$ are the \emph{support numbers} of $P$. 
\end{definition}

A $t$-convex polygon is called \emph{elementary} if it is defined by a single
future time-like vector $\eta$ and a positive number $h$. Note that 
for each $k$, $(H_t^k (h\eta))^{\bot}$ is tangent
to $h\H$ (the upper hyperbola with radius $h$).
Hence a $t$-convex polygon is the intersection of
a finite number of elementary $t$-convex polygons. 

\begin{example}\label{ex}{\rm
Let $t_0=\sinh^{-1}(1)$, so 
 $$H_{t_0}:= \left(
    \begin{array}{cc}
       \sqrt{2}  & 1 \\
   1 & \sqrt{2}
    \end{array}
\right).$$ Let us denote by $P_1$ the elementary $t_0$-convex polygon defined by the vector
$\eta= \binom{0}{1} $ and the number $h=1$, see Figure~\ref{a}.
The elementary $t_0$-convex polygon $P_2$ of Figure~\ref{b} is obtained 
from $p_1$ by a slightly change of $\eta$ and $h$. Their intersection forms
the $t_0$-convex polygon of Figure~\ref{c}. 
}
\end{example}
 \begin{figure}[htp]
   \centering
  \subfloat[A part of the $t_0$-convex polygon $P_1$. For the Lorentzian metric, all the edges have equal length and all the 
 angles between edges are equal.]{\label{a}\includegraphics[scale=0.3]{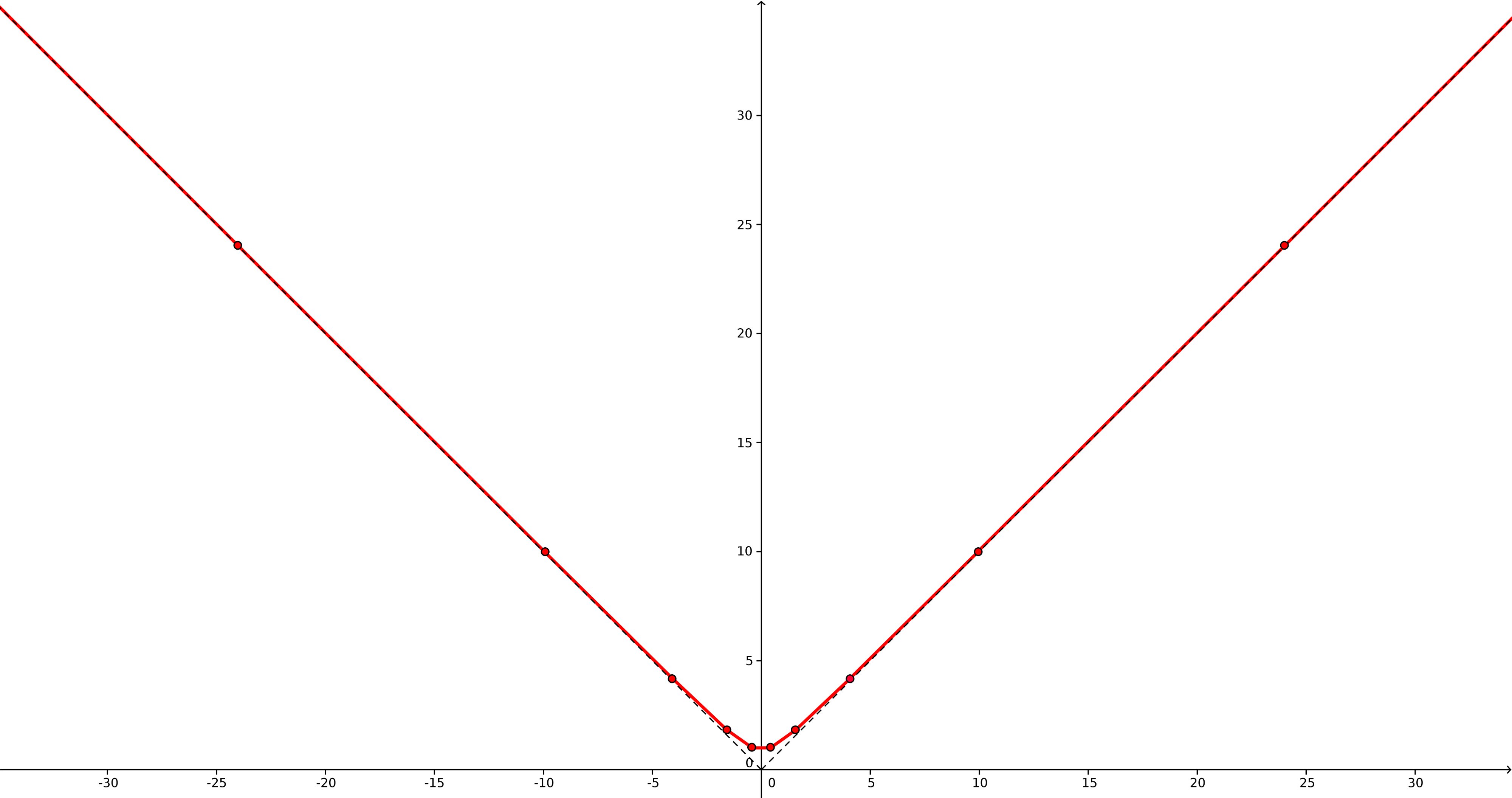} }   \\     
   \subfloat[A part of the $t_0$-convex polygon $P_2$. For the Lorentzian metric, all the edges have equal length and all the 
 angles between edges are equal.]{\label{b}\includegraphics[scale=0.3]{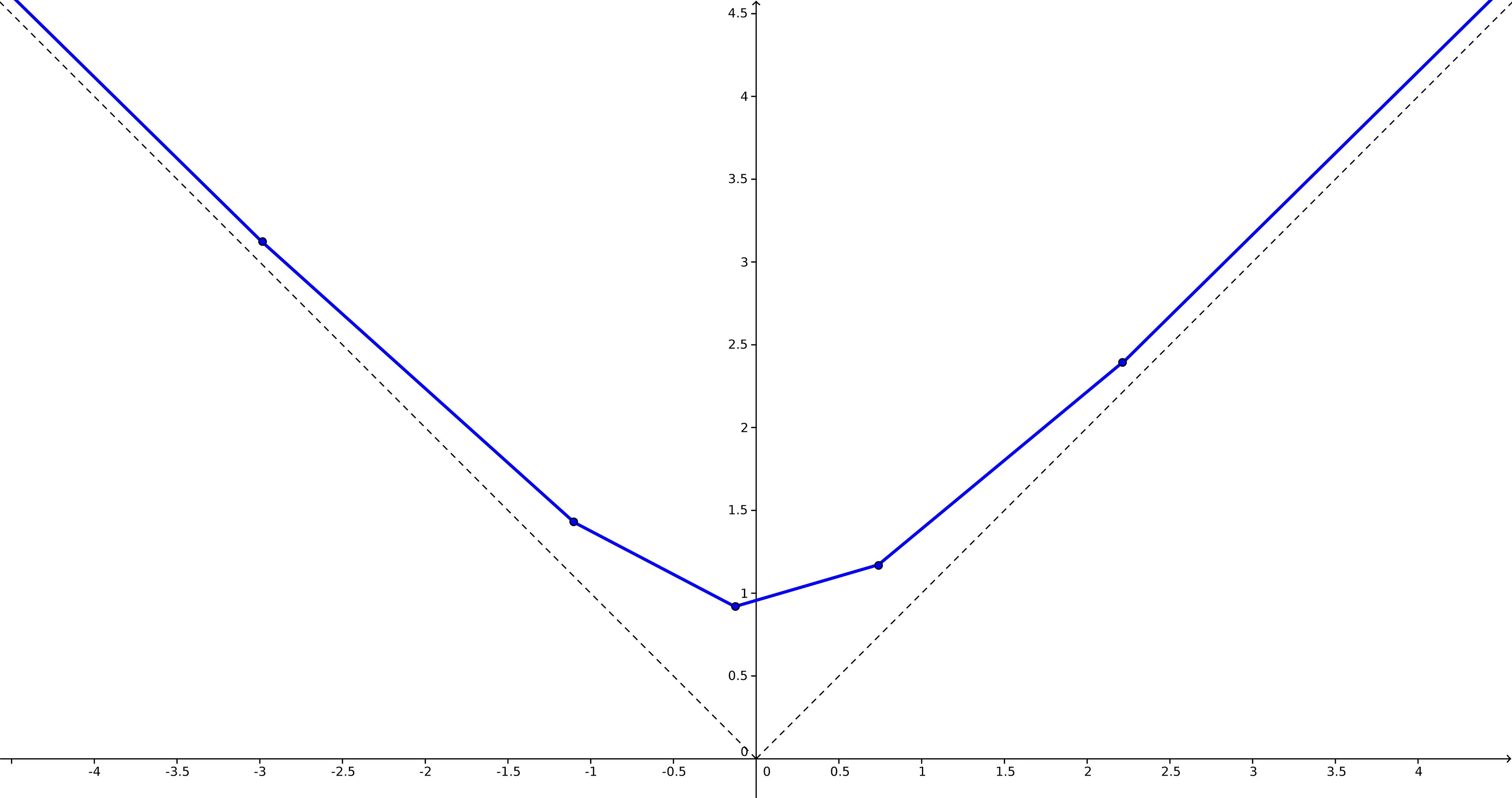}  }\\
   \subfloat[A part of the $t_0$-convex polygon obtained as the intersection of $P_1$ 
   and $P_2$.]{\label{c}\includegraphics[scale=0.3]{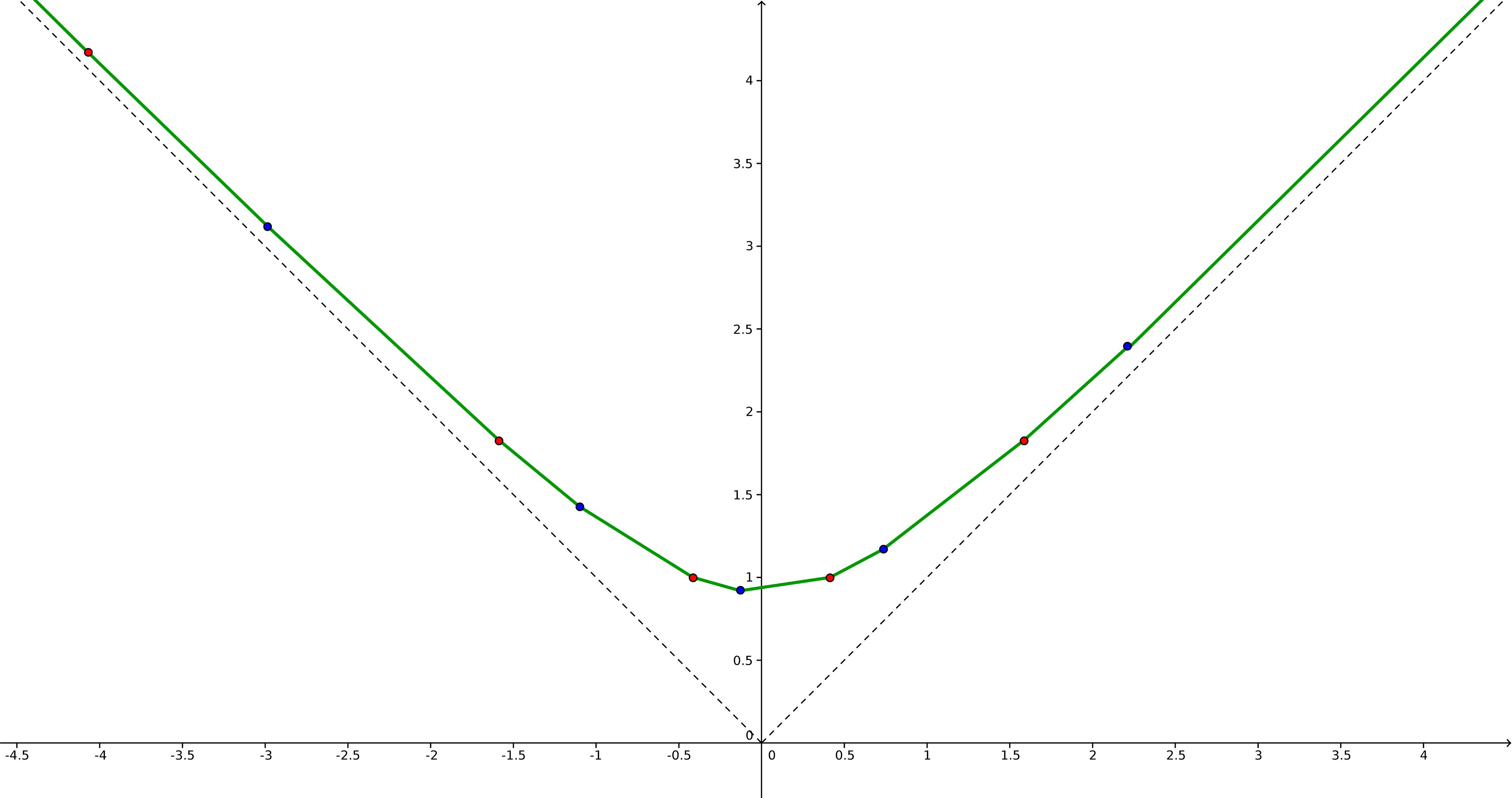}}
   \caption{To Example~\ref{ex}.}
   \label{fig:polygon}
 \end{figure}

  \begin{lemma}
  A $t$-convex polygon $P$ is a proper convex subset of $\R^2$ contained in $\F$,
  bounded by a polygonal line with a countable number of sides, and globally 
  invariant under the action of
$<H_t>$.
  \end{lemma}
  \begin{proof}
  The group invariance is clear from the definition. 
  $P$ is the intersection of a finite number of elementary $t$-convex polygons, 
  so we only have  to check the other properties in the elementary case. 
  Actually the only non-immediate one is that an elementary $t$-convex polygon
  is contained in $\F$. 
   Let us consider an elementary $t$-convex polygon made from a single future time-like vector $\eta$
   and a number $h$.
   Without
loss of generality, consider that $h=1$. 
Let $u=H^k_t(\eta)$ and $v=H^{k'}_t(\eta)$  and let $x$ be the intersection between
$u^{\bot}$ and $v^{\bot}$. As $\langle x,u \rangle_1=\langle x,v\rangle_1=-1$,  $x$ is orthogonal to
$u-v$, which is a space-like vector (compute its norm with the help of \eqref{eq:angle}). Hence $x$ is time-like, and as $u^{\bot}$ and $v^{\bot}$ never meet the 
past cone, $x$ is future. It is easy to deduce that the $t$-convex polygon is contained in $\F$.
  \end{proof}

 Note that as a convex surface, a $t$-convex polygon can also be
 a $t'$-convex polygon (for example it is also invariant under the action of any subgroup of $<H_t>$), but 
we will only consider the action of a given $<H_t>$. 

Given a  $t$-convex polygon $P$, we will  require that 
the set of elementary $t$-convex polygons such that their intersection gives $P$ is minimal, i.e~each
$\eta_i$ is the inward unit normal of a genuine edge $e_i$ of $P$.
The edge at the left (resp. right) of $e_i$ is denoted by $e_{i-1}$ (resp. $e_{i+1}$).
Let $p_i$ be the foot of the perpendicular from the origin to the line containing $e_i$ 
(in particular, $p_i=h_i\eta_i$).
Let $p_{ii+1}$ be the vertex between
$e_i$ and $e_{i+1}$. We denote by $h_{ii+1}$ (resp. $h_{ii-1}$) the signed distance
from $p_i$ to $p_{ii+1}$ (resp. from  $p_i$ to $p_{i-1i}$): it is non negative if $p_i$ is on the same side of $e_{i+1}$ (resp. $e_{i-1}$) as $P$.
The angle between $\eta_i$ and $\eta_{i+1}$ is denoted by $\varphi_i$.
See Figure~\ref{fig:notations}.

  \begin{figure}
\centering
\includegraphics[scale=0.4]{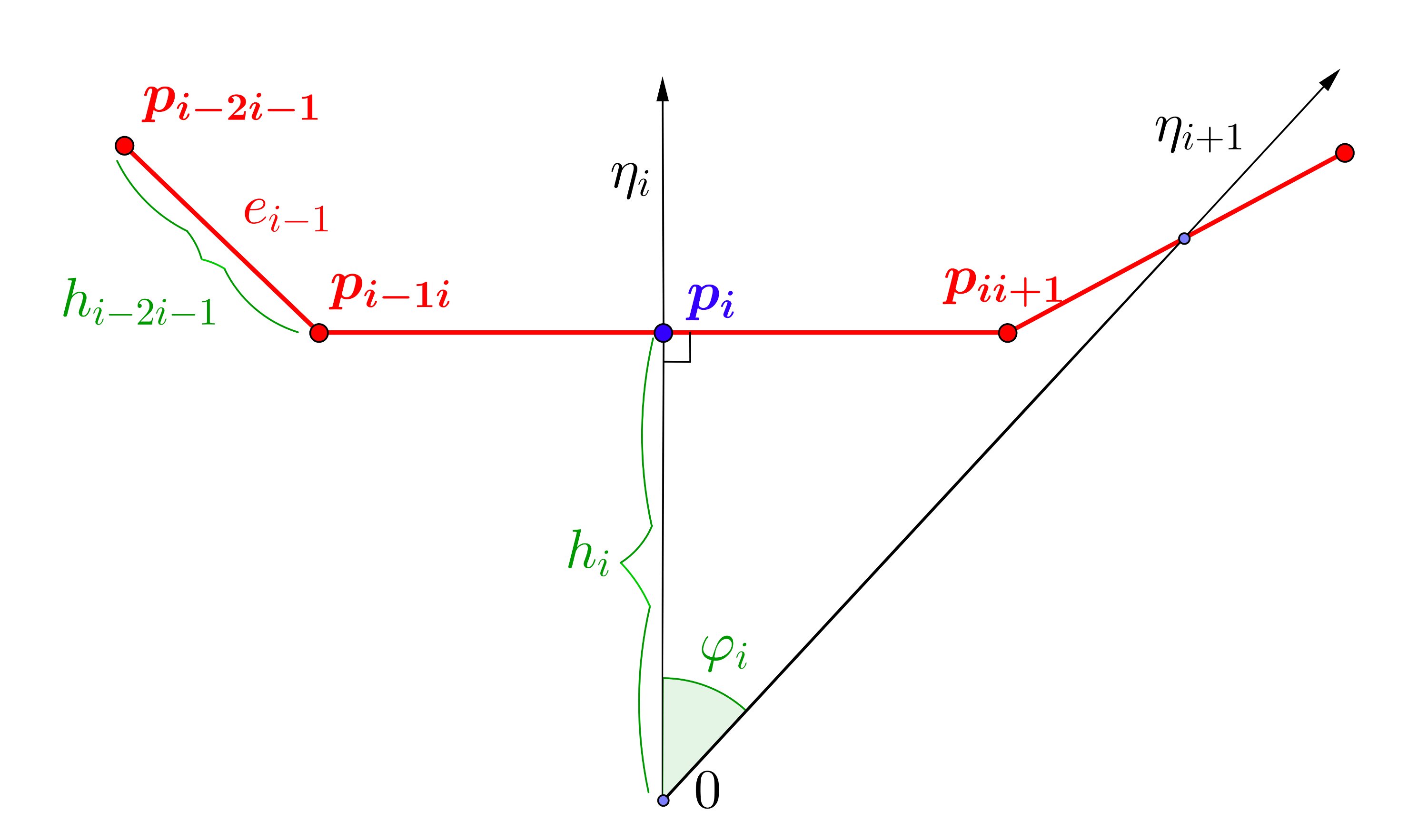}  
\caption{Notations for a $t$-convex polygon.
\label{fig:notations}
} 
  \end{figure}

\begin{lemma}\label{lem:longmink}
With the notations introduced above,
\begin{equation}\label{eq: supp num mink}
 h_{ii+1}=\frac{h_i\cosh \varphi_{i}-h_{i+1}}{\sinh \varphi_{i}},  
h_{ii-1}=\frac{h_i\cosh \varphi_{i-1}-h_{i-1}}{\sinh \varphi_{i-1}}.
\end{equation}
\end{lemma}
\begin{proof}
By definition, $h_{ii+1}$ is non negative when
$\langle p_i-p_{i+1},\eta_{i+1}\rangle_1\leq 0$,
i.e.$$-(h_{i+1}-h_i\cosh \varphi_{i})\geq 0.$$
Hence
$$h_{ii+1}= -\frac{h_{i+1}-h_i\cosh \varphi_{i}}{|h_{i+1}-h_i\cosh \varphi_{i}|}\sqrt{\langle p_{ii+1}-p_i,p_{ii+1}-p_i\rangle_1}. $$
Up to an orientation and time orientation preserving linear isometry, one can take $\eta_i=\left(0 \atop 1 \right)$.
In particular $p_i=\left(0 \atop h_i \right)$ and $(p_{ii+1})_2=h_i$  hence  
$$\langle p_{ii+1}-p_i,p_{ii+1}-p_i\rangle_1=(p_{ii+1})_1^2.$$
We also have
$\eta_{i+1}=\left(\sinh \varphi_{i} \atop \cosh  \varphi_{i} \right)$, and as
$\langle p_{ii+1},\eta_{i+1}\rangle_1=-h_{i+1}$
we get $$ (p_{ii+1})_1=\frac{-h_{i+1}+h_i\cosh\varphi_{i}}{\sinh \varphi_{i}}.$$
The proof for $h_{ii-1}$ is similar, considering $\eta_{i-1}=\left(-\sinh \varphi_{i} \atop \cosh  \varphi_{i} \right)$.
\end{proof}


\section{The cone of support vectors}

Let $P$ be a $t$-convex polygon. Choose an edge and denote its inward unit normal by $\eta_1$.
We denote the inward unit  normal of the edge on the right by $\eta_2$, and so on until 
$\eta_{n+1}=H_t(\eta_1)$. The edges with normals $\eta_1,\ldots,\eta_n$ are the 
\emph{fundamental edges} of $P$. Note that with this labeling, if $\varphi_i$ is the angle between $\eta_i$ 
and $\eta_{i+1}$, we have
\begin{equation}\label{eq:somme angle}
\varphi_{1}+\varphi_{2}+\cdots+\varphi_{n}=t.
\end{equation}

The  number $h_i(P)$ is the support number of the edge with normal $\eta_i$, 
and $h(P)=(h_1(P),\ldots,h_n(P))$ is the \emph{support vector} of $P$.
So $P$ is identified with a vector of $\R^n$, in such a way that 
$\eta_1,\ldots,\eta_n$ are in bijection with the standard basis of $\mathbb{R}^n$.
Of course $P$ is  uniquely determined by its support vector.

\begin{definition} Choose $\eta\in\H$ and
 let $\varphi_{1},\varphi_{2},\cdots,\varphi_{n}$ be positive numbers satisfying 
 \eqref{eq:somme angle}.  The \emph{cone of support vectors} 
 $\overline{\P}(\varphi_{1},\varphi_{2},\ldots,\varphi_{n})$
 is the set of support vectors of $t$-convex polygons with inward unit normals
 $\eta_1=\eta$, $\eta_{i+1}=H_{\varphi_i}(\eta_i)$.
\end{definition}

A priori the definition of $\overline{\P}$ depends not only on the angles 
 $\varphi_i$ but also on the choice of $\eta$. Actually choosing another starting $\eta'\in\H$,
 the hyperbolic translation from $\eta$ to $\eta'$ gives a linear isomorphism between the
 two resulting sets of support vectors. Hence $\overline{\P}(\varphi_{1},\varphi_{2},\ldots,\varphi_{n})$
 could be defined as the set of $t$-convex polygons with ordered angles $(\varphi_{1},\varphi_{2},\ldots,\varphi_{n})$
 up to hyperbolic translations.
 Note also that if $s$ is a cyclic permutation, then $\overline{\P}(\varphi_{s(1)},\ldots,\varphi_{s(n)})$
is the same as $\overline{\P}(\varphi_{1},\ldots,\varphi_{n})$.
 
 It is possible to prove that 
   $\overline{\P}(\varphi_{1},\varphi_{2},\ldots,\varphi_{n})$ is a  convex polyhedral cone
   with non-empty interior
   in $\R^n$, but this will be easier after a suitable metrization of $\R^n$, that is the subject of the next section.


\section{Coarea}\label{sec:coarea}

\begin{definition} Let $P\in \overline{\P}(\varphi_{1},\varphi_{2},\ldots,\varphi_{n})$.
 The \emph{coarea} of $P$ is 
$$\coarea(P)=\frac{1}{2}\sum_{i=1}^n h_i(P) \ell_i(P)$$
where the sum is on the fundamental edges, and  $\ell_i(P)=h_{ii-1}(P)+h_{ii+1}(P)$ is the length of the 
$i$th fundamental edge (hence  positive).
\end{definition}

 Geometrically
$\coarea(P)$ is the area (in the sense of the Lebesgue measure)
of a fundamental domain for the action of $H_t$ on the complement of $P$ in $\F$.
The main point is that hyperbolic translations \eqref{eq:hyp trans} have determinant $1$, so they preserve the
area, which is then independent of the choice of the fundamental domain, see Figure~\ref{fig:polygonaire}.
Moreover the area of a triangle with a space-like edge $e$ of length $l$ and $0$ as a vertex has
area $\frac{1}{2}lh$, if $h$ is the Lorentzian distance between $0$ and the line containing $e$.
(To see this, perform a hyperbolic translation such that $e$ is horizontal and compute the area.)
Note that the coarea depends not only on the polygonal line $P$ but also on the group 
$<H_t>$, so it would be more precise to speak about ``$t$-coarea'',
but as the group is fixed from the beginning, no confusion is possible. 

\begin{figure}[h]
 \centering
 \includegraphics[scale=0.3]{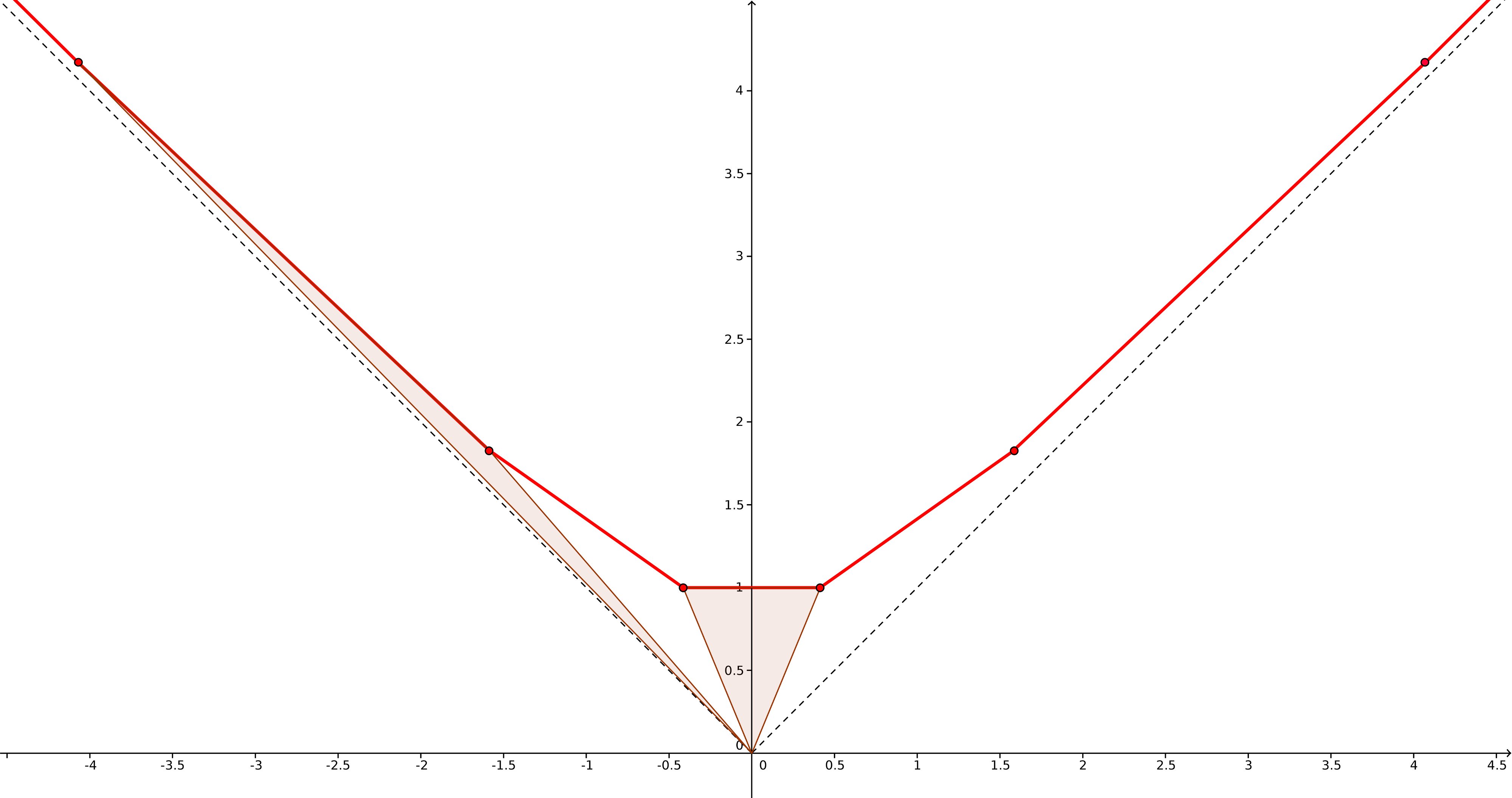}  
 \caption{The two shaded regions have the same area. This area is the coarea of the 
 polygon.
 \label{fig:polygonaire}
 } 
   \end{figure}

For a given cone of support vectors, the coarea can be formally extended 
 to $\R^n$ with the help of \eqref{eq: supp num mink}: for $h\in\R^n$,

$$\coarea(h)=\frac{1}{2}\sum_{i=1}^n h_i \ell_i(h)$$

with

\begin{equation}\label{li}\ell_i(h):=\frac{h_i\cosh \varphi_{i-1}-h_{i-1}}{\sinh \varphi_{i-1}}
+h_i \frac{h_i\cosh \varphi_{i}-h_{i+1}}{\sinh \varphi_{i}}.\end{equation}

If $n=1$, there is only one angle between the unit inward normal $\eta$
and its image under $H_t$, which is equal to $t$, and 
$\coarea(h)=h^2\frac{\cosh t -1}{\sinh t}. $

If $n\geq 2$, we introduce the \emph{mixed-coarea}
$$\coarea(h,k)=\frac{1}{2}\sum_{i=1}^n h_i \frac{k_i\cosh \varphi_{i-1}-k_{i-1}}{\sinh \varphi_{i-1}}
+h_i \frac{k_i\cosh \varphi_{i}-k_{i+1}}{\sinh \varphi_{i}},$$
which is the polarization of the $\coarea$. Actually, it is
 clearly a bilinear form, and
\renewcommand{\arraystretch}{2}
\begin{equation}\label{eq: mui}
\coarea(\eta_k,\eta_j)=\left\{
    \begin{array}{ccc}
        0 & \mbox{if } & 2\leq \vert j-k\vert\leq n+1 \\
      \displaystyle{  -\frac{1}{2}\frac{1}{\sinh\varphi_{k-1}}} & \mbox{if} & j=k-1 \\
	 \displaystyle{-\frac{1}{2}\frac{1}{\sinh\varphi_{k}}}& \mbox{if} & j=k+1 \\
 \displaystyle{\frac{1}{2}\left(\frac{\cosh \varphi_{k-1}}{\sinh\varphi_{k-1}}+\frac{\cosh\varphi_k}{\sinh\varphi_{k}}\right)}& \mbox{if} & j=k 
    \end{array}
\right.
\end{equation}
so $\coarea$ is symmetric. We also obtain the following key result.
\begin{proposition}
 The symmetric bilinear form $\coarea$ is  positive definite. 
\end{proposition}
\begin{proof}
As $\cosh \varphi_{k}>1$, the matrix 
$(\coarea(u_k,u_j))_{kj}$ is  strictly diagonally dominant, and symmetric with positive diagonal entries, hence positive definite, see for example  \cite[1.22]{Var00}.
\end{proof}

The Cauchy--Schwarz inequality 
applied to support vectors of $t$-convex polygons gives the following \emph{reversed Minkowski inequality}:

\begin{corollary} Let $P,Q$ be $t$-convex polygons with parallel edges. Then
 $$\coarea(P,Q)^2\leq \coarea(P)\coarea(Q),$$
with equality if and only if $P$ and $Q$ are homothetic: $\exists \lambda>0, \forall i, h_i(P)=\lambda h_i(Q)$.
\end{corollary}


\section{Spherical orthoschemes}

 $\overline{\P}(\varphi_{1},\varphi_{2},\ldots,\varphi_{n})$ is clearly a cone in $\R^n$.
 Moreover it is the set of vectors of positive edge lengths, for the edge lengths defined by \eqref{li}. 
 From the definition of the coaera, for $h\in\R^n$, $2 \coarea(\eta_i,h)=\ell_i(h)$, so
 $\eta_i$  is an inward normal vector to the 
facet of $\overline{\P}$ defined by $\ell_i=0$. So $\overline{\P}$ is polyhedral, and it is convex
because the $\eta_i$ form a basis of $\R^n$. 
 Let us denote by $\P(\varphi_{1},\varphi_{2},\ldots,\varphi_{n})$
 the intersection of  $\overline{\P}(\varphi_{1},\varphi_{2},\ldots,\varphi_{n})$
 with the unit sphere of $(\R^n,\coarea)$ (i.e.~the set of support vectors of $t$-convex polygons with
 coarea one). It follows that $\P$ is a spherical simplex.
 If $n=1$, $\P$ is a  point on a line, so from now on assume that $n>1$.

When $n=2$, $\P$ is an arc on the  unit circle with length $\theta$ satisfying
$$\cos\theta=\frac{\sinh\varphi_{2}}{\sinh(\varphi_1+\varphi_2)}.$$

When  $n=3$, $\P$ is a spherical triangle with acute inner angles, whose cosines are 
given by:
\begin{equation}\label{formcos}-\frac{\coarea(\eta_k,\eta_{k+1})}{\sqrt{\coarea(\eta_k,\eta_k)}\sqrt{\coarea(\eta_{k+1},\eta_{k+1})}}=\sqrt{\frac{\sinh \varphi_{k-1}\sinh\varphi_{k+1}}{\sinh(\varphi_{k-1}+\varphi_{k})
\sinh(\varphi_{k}+\varphi_{k+1})}}.\end{equation}

When $n\geq 3$, from \eqref{eq: mui} we see that each facet 
has an acute interior dihedral angle with exactly two other facets, and is orthogonal to the other facets.
Such spherical simplexes are called  \emph{acute spherical orthoschemes}. See \cite[5]{Deb90} for the history and main properties
of these very particular simplexes. Note that there are no spherical Coxeter orthoschemes, because the Coxeter diagram of a 
spherical orthoscheme must be a cycle, and there
is no cycle in the list of Coxeter diagrams of spherical Coxeter simplexes. 
The list can be found for example in \cite{Rat06}.

Let us denote by $U_k$ the line through $p_k$ (so the angle between $U_k$ and $U_{k+1}$ is $\varphi_k$), and by
$\lambda$ the cross ratio $[U_{k-1},U_k,U_{k+1},U_{k+2}]$, namely if $u_{k-1},u_k,u_{k+1},u_{k+2}$
are the intersections of the lines $U_i$ with any line not passing through zero and endowed with coordinates then
(see \cite{Ber})
$$\lambda=[U_{k-1},U_k,U_{k+1},U_{k+2}]=\frac{u_{k+1}-u_{k-1}}{u_{k+1}-u_k}\frac{u_{k+2}-u_k}{u_{k+2}-u_{k-1}}.$$
We have the formula (see \cite{PY})
$$\frac{\sinh \varphi_{k-1}\sinh\varphi_{k+1}}{\sinh(\varphi_{k-1}+\varphi_{k})
\sinh(\varphi_{k}+\varphi_{k+1})}=\frac{\lambda-1}{\lambda}=[U_{k-1},U_{k+2},U_k,U_{k+1}]. $$

From a given  $n$-dimensional  acute spherical orthoscheme $O$    we can find angles (positive real numbers) $(\varphi_{1},\varphi_{2},\ldots,\varphi_{n})$
such that  $\P(\varphi_{1},\varphi_{2},\ldots,\varphi_{n})$ is isometric to $O$.
Let $0<A<1$ be the square of the cosine of an acute dihedral angle of $O$. 
We have first to find  ordered time-like lines $U_1,U_2,U_3,U_4$ such that 
$[U_1,U_2,U_3,U_4]=\frac{1}{1-A}$, i.e.~we have to prove that the cross-ratio of the lines can reach any 
value $>1$. Choose arbitrary distinct ordered time-like $U_1,U_2, U_4$. If $U_3=U_4$ then $[U_1,U_2,U_3,U_4]=1$, and if $U_3=U_2$ then  $[U_1,U_2,U_3,U_4]=+\infty$, so by continuity
any given value $>1$ can be reached for a suitable $U_3$ between $U_2$ and $U_4$. 
$U_1,U_2,U_3,U_4$ give angles $\varphi_1,\varphi_2,\varphi_3$.

Now the other $\varphi_k$ are easily obtained as follows.
Given the next dihedral angle of $O$ (they can be ordered by ordering the unit normals to $O$, see \cite{Deb90}),
the square of its cosine should be equal to
$$\frac{\sinh \varphi_{2}\sinh\varphi_{4}}{\sinh(\varphi_{2}+\varphi_{3})
\sinh(\varphi_{3}+\varphi_{4})}
$$
and $\varphi_2,\varphi_3$ are known, so we get $\varphi_4$. And so on.

\section{Spherical cone-manifolds}

Let $n>2$ and consider the orthoscheme  $\P=\P(\varphi_1,\ldots,\varphi_n)$.
A facet of $\P$ is isometric to the space of $t$-convex polygons with $\eta_1,\ldots,\hat{\eta_i},\ldots,\eta_n$ ($\hat{\eta_i}$ means that $\eta_i$ is deleted from the list) 
as normals to the fundamental edges. The angles between the normals are 
$\varphi_1,\ldots,\varphi_{i-2},\varphi_{i-1}+\varphi_{i},\varphi_{i+1},\ldots,\varphi_n$.
This orthoscheme is also isometric to a facet of the orthoscheme  $\P'$ obtained
by permuting $\varphi_{i-1}$ and $\varphi_i$ in the list of angles. Hence we can glue $\P$ and $\P'$
isometrically
along this common facet.
We denote by $\C(\varphi_1,\ldots,\varphi_n)$ the $(n-1)$-dimensional spherical cone-manifold obtained by gluing 
in this way all the $(n-1)!$ orthoschemes obtained by permutations of the list  $\varphi_1,\ldots,\varphi_n$, up to cyclic permutations.

When $n=3$, $\C(\varphi_1,\varphi_2,\varphi_3)$ is isometric to a spherical cone-metric on the sphere with
three conical singularities, with cone-angles $<\pi$, obtained by gluing two isometric spherical triangles along corresponding edges.

Let $n\geq 4$. Around the codimension 2 face of $\C$ isometric to
$$N:=\mathcal{C}(\varphi_1,\ldots,\varphi_k+\varphi_{k+1},\ldots,\varphi_j+\varphi_{j+1},\ldots,\varphi_{n+3}) $$
 are glued four orthoschemes, corresponding to the four ways of ordering $(\varphi_k,\varphi_{k+1})$ and $(\varphi_j,\varphi_{j+1}).$
 As the dihedral angle of each orthoscheme at such codimension $2$ face is $\pi/2$, the total angle around $N$ 
in $\C$ is $2\pi$. Hence metrically $N$ is actually not a singular set.
Around the codimension 2 face of $\C$ isometric to
$$S:=\C(\varphi_1,\ldots,\varphi_k+\varphi_{k+1}+\varphi_{k+2},\ldots,\varphi_{n+3}) $$
 are glued six orthoschemes corresponding to the six ways of ordering $(\varphi_k,\varphi_{k+1},\varphi_{k+2})$.
Let $\Theta$ be the cone-angle around $S$.
It is the sum of the dihedral angles of the six orthoschemes glued around it. As formula~\eqref{formcos} 
is symmetric for two variables, $\Theta$ is two times the sum of three different dihedral angles. A direct computation
gives ($k=1$ in the formula)
$$\cos(\Theta/2)=\textstyle 
\frac{\sinh\varphi_1\sinh\varphi_2\sinh\varphi_3-\sinh(\varphi_1+\varphi_2+\varphi_3)(\sinh\varphi_1\sinh\varphi_2+
\sinh\varphi_2\sinh\varphi_3+\sinh\varphi_3\sinh\varphi_1)}{\sinh(\varphi_1+\varphi_2)\sinh(\varphi_2+\varphi_3)\sinh(\varphi_3+\varphi_1)}. $$

During the computation we used that
$$\sinh(a+b)\sinh(b+c)-\sinh a \sinh c=\sinh b\sinh(a+b+c)  $$
which can be checked with
$\frac{1}{2}\left(\cosh(x+y)-\cosh(x-y) \right)=\sinh x\sinh y$. The analogous formula in the Euclidean convex polygons case
was obtained in \cite{kojimaal1}.

For example when $\varphi_i=\varphi \;\forall i$, we have
$$\cos(\Theta /2)=-\frac{2\cosh(\varphi)^2+\sinh(\varphi)^2}{2\cosh(\varphi)^3}. $$
The function on the right-hand side is a bijection from the positive numbers to $]-1,0[$, hence all
 the $\Theta\in ]2\pi,3\pi[$ (the dihedral angle $\theta\in ]\pi/3,\pi/2[$) are uniquely reached. 
In particular $\C$ is not an orbifold.

The cone-manifold $\C$ comes with an isometric involution which consists of reversing the order of the angles 
$(\varphi_1,\ldots,\varphi_n)$.

\section{Higher dimensional generalization}\label{sec:gen}

The generalization of $t$-convex polygons to higher dimensional Minkowski spaces is as follows.
Let us consider the $d$-dimensional hyperbolic space $\H^d$ as a pseudo-sphere in the $d+1$-dimensional
Minkowski space $M^{d+1}$, and let $\Gamma$ be a discrete group of linear isometry of $M^{d+1}$ such that
$\H^d/\Gamma$ is a compact hyperbolic manifold. A $\Gamma$-convex polyhedron is,
given $\eta_1,\ldots,\eta_n\in\H^d$ and positive numbers $h_1,\ldots,h_n$, the intersection of the
future sides of the space-like hyperplanes $(\gamma(h_i\eta_i))^{\bot}$ $\forall i, \forall \gamma\in\Gamma$.
The mixed-coarea is generalized as a ``mixed covolume''. For details and computation of the signature,
see \cite{FF}. Actually for a given set of $\eta_i$, many combinatorial types may appear, and one has to restrict 
to type cones (cones of polyhedra with parallel facets and same combinatorics). It should be
interesting to investigate the kind of spherical polytopes that  appear.

Another related question is to look at the quadratic form given by the face area 
of the polyhedra (in a fundamental domain) and its relations 
with the moduli spaces of flat metric with conical singularities of negative curvature on compact
surfaces of genus $>1$
(the quotient of the boundary of a $\Gamma$-convex
polyhedron is isometric to such a metric).

The analogous questions in the  convex polytopes  case  are the subject of \cite{FI}. The moduli space of
flat metrics on the sphere was studied in \cite{T}.

\section*{Acknowledgement}

The author thanks anonymous referee and Haruko Nishi who helped to imporve the redaction of the present text.
Up to trivial changes, the introduction was written by an anonymous referee. 
The polygons introduced in the present paper are very particular cases of objects studied in 
\cite{FF} and \cite{FV}.

Work supported by the ANR GR Analysis-Geometry.


\end{document}